\documentclass[12pt,abstract=true]{scrartcl}
\usepackage{amsmath,amssymb,amsthm,url}
\usepackage{authblk}

\DeclareMathAlphabet{\curly}{U}{rsfs}{m}{n}

\newtheorem{thm}{Theorem}[section]

\newtheorem{lem}[thm]{Lemma}

\newtheorem{prop}[thm]{Proposition}

\newtheorem*{theorem*}{Theorem}
\theoremstyle{remark}\newtheorem*{remark}{Remark}

\numberwithin{equation}{section}

\newcommand{\1}{\mathbf{1}}

\newcommand{\leg}[2]{\genfrac{(}{)}{}{}{#1}{#2}}
\newcommand\Li{\mathrm{Li}}

\newcommand\F{\mathbb{F}}
\newcommand\Z{\mathbb{Z}}
\newcommand\N{\mathbb{N}}
\newcommand\Q{\mathbb{Q}}

\newcommand\lcm{\mathrm{lcm}}
\newcommand\Gal{\mathrm{Gal}}
\newcommand\A{\curly{A}}
\newcommand\C{\mathcal{C}}
\newcommand\Ll{\curly{L}}
\newcommand\Pp{\curly{P}}
\newcommand\Qq{\curly{Q}}

\newcommand\E{\mathbb{E}}
\newcommand\Prob{\mathbb{P}}
\newcommand{\legq}[2]{\genfrac{[}{]}{}{}{#1}{#2}}
\renewcommand{\phi}{\varphi}
\newcommand{\rad}{\mathrm{rad}}
\makeatletter
\renewcommand{\pod}[1]{\mathchoice
  {\allowbreak \if@display \mkern 18mu\else \mkern 8mu\fi (#1)}
  {\allowbreak \if@display \mkern 18mu\else \mkern 8mu\fi (#1)}
  {\mkern4mu(#1)}
  {\mkern4mu(#1)}
}
\makeatletter
\newcommand{\subjclass}[2][2010]{%
  \let\@oldtitle\@title%
  \gdef\@title{\@oldtitle\footnotetext{#1 \emph{Mathematics subject classification.} #2}}%
}
\newcommand{\keywords}[1]{%
  \let\@@oldtitle\@title%
  \gdef\@title{\@@oldtitle\footnotetext{\emph{Key words and phrases.} #1.}}%
}
\makeatother

\begin{document}
\title{Bounded gaps between primes with a given primitive root, II}

\author{Roger C. Baker%
  \thanks{email: \texttt{baker@math.byu.edu}}}
\affil{Department of Mathematics\\Brigham Young University\\Provo, UT 84602, USA}

\author{Paul Pollack%
  \thanks{email: \texttt{pollack@math.uga.edu}}}
\affil{Department of Mathematics\\ University of Georgia\\Athens, GA 30602, USA}

\date{}

\subjclass{Primary: 11A07, Secondary: 11G05}
\keywords{Artin's conjecture, bounded gaps, elliptic Artin, Maynard--Tao theorem, primitive root}
\maketitle

\begin{abstract}\noindent Let $m$ be a natural number, and let $\Qq$ be a set containing at least $\exp(C m)$ primes. We show that one can find infinitely many strings of $m$ consecutive primes each of which has some $q\in\Qq$ as a primitive root, all lying in an interval of length $O_{\Qq}(\exp(C'm))$. This is a bounded gaps variant of a theorem of Gupta and Ram Murty. We also prove a result on an elliptic analogue of Artin's conjecture. Let $E/\Q$ be an elliptic curve with an irrational $2$-torsion point. Assume GRH. Then for every $m$, there are infinitely many strings of $m$ consecutive primes $p$ for which $E(\F_p)$ is cyclic, all lying an interval of length $O_E(\exp(C'' m))$. If $E$ has CM, then the  GRH assumption can be removed. Here $C$, $C'$, and $C''$ are absolute constants.
\end{abstract}

\section{Introduction}
In 1927, Artin proposed the following conjecture: \emph{If $g$ is not a square and $g\neq -1$, then there are infinitely many primes $p$ for which $g$ is a primitive root modulo $p$.} Artin's conjecture remains unsolved, but investigations in this direction  have led to many deep and beautiful results (see \cite{moree12}).

In 1967, Hooley \cite{hooley67} showed that Artin's conjecture is a consequence of the Generalized Riemann Hypothesis for Dedekind zeta functions (hereafter GRH). In \cite{pollack14}, it was shown how Hooley's proof could be merged with the method of Maynard--Tao for producing bounded gaps between primes: \emph{On GRH, for every nonsquare $g\neq -1$ and every $m$, there are infinitely many runs of $m$ consecutive primes all possessing $g$ as a primitive root and lying in an interval of length $O_m(1)$.}

There is not a single $g$ for which the conclusion of Artin's conjecture is known to hold unconditionally. However, in 1984 Gupta and Ram Murty \cite{GM84} described how to produce many finite sets of integers some member of which satisfies Artin's conjecture. Their method was refined by Ram Murty and Srinivasan \cite{MS87}, Gupta, Ram Murty, and Kumar Murty \cite{GMM87}, and by Heath-Brown \cite{HB86}. It follows from the results in this last paper that Artin's conjecture holds for at least one $g \in\{2, 3, 5\}$. We prove a result in this direction where the primes produced are consecutive and contained in an interval of bounded length.

Recall that nonzero $q_1, \dots, q_r\in \Z$ are said to be \emph{multiplicatively independent} if $q_1^{e_1} \cdots q_r^{e_r}=1$ in integers $e_1, \dots, e_r$ only when $e_1 = \dots = e_r = 0$.

\begin{thm}\label{thm:main} Let $\Qq$ be a set of $r$ multiplicatively independent integers. Assume that the elements $q_1, \dots, q_r$ of $\Qq$ satisfy the following technical condition:
\[\tag{*} \parbox{0.8\linewidth}{If $e_0, e_1, \dots, e_r $ are nonnegative integers for which $(-3)^{e_0} q_1^{e_1}\cdots q_r^{e_r}$ is a square, then $\sum_{i=0}^{r} e_i$ is even.}\]
Let $m$ be a natural number. If $r \ge \exp(C m)$, then there are infinitely many runs of $m$ consecutive primes $p_1 < \dots < p_m$ all of which possess some element of $\Qq$ as a primitive root, where also
\[ p_m - p_1 \leq \mathfrak{f}(\Q(\sqrt{q_1}, \dots, \sqrt{q_r})/\Q) \cdot \exp(C'm).\]
Here $C$ and $C'$ are (positive) absolute constants, and $\mathfrak{f}(K/\Q)$ denotes the conductor of the abelian extension $K/\Q$.
\end{thm}

\begin{remark} Of course, (*) holds whenever $q_1, \dots, q_r$ are distinct (positive) primes.\end{remark}

The techniques used to attack Artin's conjecture can also be used to answer statistical questions about reductions of elliptic curves. Here the general setup is as follows:
Let $E/\Q$ be an elliptic curve. For all but finitely many primes $p$, one can reduce $E$ mod $p$ to obtain an elliptic curve defined over $\F_p$. What can one say about the structure of the group $E(\F_p)$ as $p$ varies? It is known that $E(\F_p)$ is always generated by two elements, and so it is particularly natural to ask when one suffices. In other words, how often is $E(\F_p)$ a cyclic group?

If all of the $2$-torsion of $E$ is defined over $\Q$, then $(\Z/2\Z)^2$ sits inside $E(\Q)$, and so $E(\F_p)$ is cyclic for at most finitely many primes $p$. So assume $E$ has an irrational $2$-torsion point. Assuming GRH, Serre showed that there are infinitely many primes $p$ with $E(\F_p)$ cyclic, using Hooley's approach \cite{hooley67} to Artin's conjecture. In fact, Serre \cite{serre79} obtained an asymptotic formula for the number of such $p \leq x$, as $x \to \infty$. Ram Murty \cite{murty83} showed that when $E$ has CM, Serre's asymptotic formula can be proved unconditionally; a simpler argument for the same conclusion has been given by Cojocaru \cite{cojocaru03}. See \cite{CM04} and \cite{AM10} for investigations into the size of the error term in Serre's formula.

We prove the following bounded gaps result.

\begin{thm}\label{thm:EC} Let $E/\Q$ be an elliptic curve with an irrational $2$-torsion point. Let $m$ be a natural number. If GRH holds, then there are infinitely many runs of $m$ consecutive primes $p_1 < p_2 < \dots < p_m$ for which $E(\F_p)$ is cyclic, where
\[ p_m-p_1 \leq \rad(\Delta_{E}) \cdot \exp(C'' m). \]
Here $\rad(\Delta_{E})$ is the product of the primes of bad reduction, and $C''$ is an absolute constant. If $E$ has $CM$, then the GRH assumption can be removed.
\end{thm}

The CM case of Theorem \ref{thm:EC} is particularly easy because of the abundance of supersingular primes. According to a criterion of Deuring (see, e.g.,  \cite[Theorem 12, p. 182]{lang87}), a prime $p$ of good reduction is supersingular precisely when there is a unique prime in the CM field lying above $p$. As we explain below, this implies that $\E(\F_p)$ is cyclic for all primes from a certain arithmetic progression. This allows us to appeal to a recent theorem of Banks--Freiberg--Turnage-Butterbaugh \cite{BFTB14} about long runs of such primes in short intervals.

It is perhaps slightly unsettling that we produce only supersingular primes in the CM case. In general, this is unavoidable. For instance, consider the curve $E$ given by $y^2=x^3+x$, whose $2$-torsion points are defined over $\Q(i)$. Since $E$ has CM by $\Z[i]$, Deuring's criterion tells us that a prime $p$ of good ordinary reduction splits in $\Q(i)$, and so $E(\F_p)$ contains $(\Z/2\Z)^2$ for all such $p$. Our final theorem says that \emph{if} there are infinitely many $p$ of good ordinary reduction with $E(\F_p)$ cyclic, then the set of these $p$ has bounded gaps.

\begin{thm}\label{thm:CM} Let $E/\Q$ be a CM elliptic curve. Assume that there are infinitely many primes $p$ of good ordinary reduction for which $E(\F_p)$ is cyclic. Then there are infinitely many tuples of $m$ such primes $p_1 <\dots < p_m$ with $p_m-p_1 \ll \exp(O_E(m))$.
\end{thm}

\noindent Unfortunately, the method of proof of Theorem \ref{thm:CM} does not allow us to impose the condition that the primes produced here are consecutive.

It would be desirable to remove the GRH assumption altogether from Theorem \ref{thm:EC}. We note that in \cite{GM90}, Gupta and Ram Murty showed unconditionally that if $E$ has an irrational $2$-torsion point, then there are always infinitely many primes $p$ with $E(\F_p)$ cyclic (but they do not get the order of magnitude for the count predicted by Serre's asymptotic formula). Their proof relies on a sieve result seemingly unavailable in our context.

\subsection*{Notation} The letters $\ell$ and $p$ are reserved for primes. We write $p^{-}(n)$ for the smallest prime factor of $n$, with the convention that $p^{-}(1)=\infty$. We use $\rad(n)$ to denote the largest squarefree divisor of $n$.  We use $C_1, C_2, \dots$ for absolute positive constants that are be thought of as large. If $F$ is a number field, $\Z_F$ denotes its ring of integers, and we write $\Delta_F$ for the absolute discriminant of $F$. If $E$ is an elliptic curve defined over $\Q$, we let $\Delta_E$ denote the minimal discriminant of $E/\Q$.  We write $\Prob(\cdot)$ for the probability of an event and $\E[\cdot]$ for the expectation of a random variable.

\section{Preliminaries for the proof of Theorem \ref{thm:main}}

If $q_1, \dots, q_r \in \Z$ and $p\nmid q_1 \cdots q_r$, we write $\langle q_1, \dots, q_r\bmod{p}\rangle$ for the subgroup of $\F_p^{\times}$ generated by the mod $p$ reductions of the $q_i$. The next lemma is due to Ram Murty and Srinivasan \cite{MS87} (compare with \cite[Lemma 2]{GM84}).

\begin{lem}\label{lem:MS} Let $q_1, \dots, q_r$ be multiplicatively independent integers, and let $Y \ge 1$. The number of primes $p$ for which
\[ \#\langle q_1, \dots, q_r \bmod{p}\rangle \leq Y \]
is $O(Y^{1+ \frac{1}{r}})$. Here the implied constant may depend on the $q_i$.
\end{lem}
\begin{proof} We include the short proof. Suppose that $\#\langle q_1, \dots, q_r \bmod{p}\rangle \leq Y$. By the pigeonhole principle, as $e_1, \dots, e_r$ run independently from $0$ through $\lfloor Y^{1/r}\rfloor$, two expressions of the form $q_1^{e_1} \cdots q_r^{e_r}$ must coincide mod $p$. Consequently, for some choice of integers $e_i'$ with each $|e_i'| \le Y^{1/r}$ and not all $e_i'=0$,
$p$ divides the numerator of the nonzero rational number $q_1^{e_1'} \cdots q_r^{e_r'}-1$. This numerator is (crudely) bounded above by $2\max\{|q_1|, \dots, |q_r|\}^{rY^{1/r}}$ and so has $O(Y^{1/r})$ prime divisors. Summing over the $O(Y)$ possibilities for the $e_i'$ completes the proof.
\end{proof}

The following lemma is used to construct an admissible collection of linear functions to which Maynard's machinery can be applied.

\begin{lem}\label{lem:HB} Let $q_1, \dots, q_r$ be nonzero integers satisfying {\normalfont\rmfamily{(*)}}. Let $v = 16 \prod_{\ell \mid q_1\cdots q_r,~\ell>2} \ell$. One can select an integer $u$ coprime to $v$ so that both of the following hold:
\begin{enumerate}
\item[\normalfont\rmfamily{(1)}] For every $p\equiv u\pmod{v}$, the Legendre symbols $\leg{q_1}{p}=\dots=\leg{q_r}{p}=-1$.
\item[\normalfont\rmfamily{(2)}] If $T$ is the largest power of $2$ dividing $u-1$, then $T \in \{2, 4, 8\}$, and $\gcd(\frac{u-1}{T},v)=1$.
    \end{enumerate}
\end{lem}
\begin{proof} For $r=3$, this lemma was proved by Heath-Brown. Since the argument for the general case is the same, we only outline the main steps, referring the reader to \cite[pp. 35--36]{HB86} for the details. By estimating $\sum_{p \le x}\left(1-\leg{-3}{p}\right)\prod_{i=1}^{r} \left(1-\leg{q_i}{p}\right)$ from below --- keeping (*) in mind ---  one shows that there are infinitely many primes $p$ with $\leg{-3}{p} = \leg{q_1}{p} = \dots = \leg{q_r}{p}=-1$. Fix one and call it $p_0$. For each odd prime $\ell$ dividing $q_1 \cdots q_r$, put
\[ u_{\ell} = \begin{cases} p_0 &\text{if $\ell \nmid p_0-1$,}\\
4p_0 &\text{otherwise},
\end{cases}
\quad\text{and put}\quad
u_2 = \begin{cases} p_0 &\text{if $16 \nmid p_0-1$}, \\
p_0-8 &\text{otherwise}.
\end{cases}
\]
Then for all odd primes $\ell \mid q_1\cdots q_r$, we see that $\ell \nmid u_{\ell}-1$. (We have used here that $p_0 \equiv -1\pmod{6}$, since $\leg{-3}{p}=-1$.) One checks that it suffices to choose $u$ as a solution to the simultaneous congruences
\[ u \equiv u_{\ell} \pmod{\ell}~\forall \text{ odd $\ell \mid q_1\dots q_r$} \quad\text{and} \quad u \equiv u_2 \pmod{16}. \qedhere \]
\end{proof}

Let $\Ll$ be a set of $k$ distinct linear functions, say $L_1(n) = a_1 n + b_1, \dots, L_k(n)= a_k n + b_k$, where each $a_i, b_i \in \Z$ and every $a_i>0$. We say that $\Ll$ is \emph{admissible} if for each prime $p$, there is some integer $n_p$ for which $p \nmid \prod_{i=1}^{k} L_i(n_p)$. Note that if each $(a_i, b_i)=1$, to check admissibility it suffices to check primes $p \leq k$.

\begin{lem}\label{lem:admissible} Let $q_1, \dots, q_r$ be nonzero integers satisfying {\normalfont\rmfamily{(*)}}, and let $u$ and $v$ be chosen as in Lemma \ref{lem:HB}. Let $\kappa$ be a natural number. There are integers $a_1 < \dots <  a_{\kappa}$, each congruent to $u\bmod{v}$, for which the $2\kappa$ linear functions
\begin{alignat*}{3} L_1(n) &= vn + a_1, &\quad \dots, \quad &L_{\kappa}(n) = vn + a_{\kappa}, \\
\tilde{L}_1(n) &= \frac{v}{T}n + \frac{a_1-1}{T}, &\quad \dots, \quad &\tilde{L}_{\kappa}(n) = \frac{v}{T} n + \frac{a_{\kappa}-1}{T} \end{alignat*}
make up an admissible family. Moreover, we can select the $a_{i}$ in such a way that
\[ a_{\kappa} - a_1 \leq v\cdot (2\kappa)^{C_1}. \]
\end{lem}
\begin{proof} By the fundamental lemma of the sieve, if $C_1$ is large enough, then the number of integers $A \in [0,(2\kappa)^{C_1}]$ for which $p^{-}((vA+u)(\frac{v}{T}A+\frac{u-1}{T})) > 2\kappa$ exceeds \[ \frac{1}{2} ((2\kappa)^{C_1}) \prod_{p \leq 2\kappa} (1-2/p).\]
Increasing $C_1$ if necessary, this lower bound exceeds $\kappa$. Pick $\kappa$ of these integers, say $A_1 < \dots< A_{\kappa}$. The theorem follows upon choosing $a_i = vA_i + u$. Indeed, for primes $p \leq 2\kappa$, we have arranged matters so that $p \nmid \prod_{i=1}^{\kappa} L_{i}(0) \tilde{L}_{i}(0)$.
\end{proof}

\begin{remark} In the next section, we will show that all of the $\tilde{L}_i$ are almost primes at the same time that several of the ${L}_i$ are prime. A similar strategy appears in work of Li and Pan \cite{LP14}, who seem to have been the first to notice that the Maynard--Tao method can be applied with auxiliary `almost prime' conditions added. In the context of the earlier GPY method, this observation was made by Pintz \cite{pintz10}.
\end{remark}

\section{Proof of Theorem \ref{thm:main}}
The following key proposition is contained in recent work of Maynard \cite{maynard14d}.

\begin{prop}\label{prop:maynard} Fix an admissible family $\Ll$ of $k$ distinct linear functions, where $k \ge 2$. Suppose that $x$ is sufficiently large, $x > x_0(\Ll)$. There is a probability measure on
\[ \A(x):= \{n \in \Z: x \leq n < 2x\} \]
with all of the following properties:
\begin{enumerate}
\item[\normalfont\rmfamily{(1)}] The probability mass at any single $n \in \A(x)$ is \[ \ll x^{-1} (\log{x})^k \left(\prod_{i=1}^{k} \prod_{p \mid L_i(n)} 4\right) \exp(O(k\log{k})). \]
\item[\normalfont\rmfamily{(2)}] For each $L \in \Ll$,
\[ \Prob(L(n)\text{ is prime})\gg \frac{\log{k}}{k}. \]
\item[\normalfont\rmfamily{(3)}] Suppose that $\rho \in [k\frac{(\log\log{x})^2}{\log{x}}, \frac{1}{25}]$. For each $L \in \Ll$,
    \[ \E\bigg[\sum_{\substack{p \mid L(n) \\ p < x^{\rho}}} 1\bigg] \ll \rho^2 k^4 (\log{k})^2. \]
\item[\normalfont\rmfamily{(4)}] Suppose that $L(n)=a_0n+b_0$ is a linear function not belonging to $\Ll$. Suppose also that $|a_0|, |b_0| \le x$ and that $\Delta_L$, defined by
    \[ \Delta_L:= a_0 \prod_{j=1}^{k} |a_0 b_j - b_0 a_j|, \]
    is nonzero. Then
    \[ \Prob(p^{-}(L(n)) > x^{1/25}) \ll \frac{\Delta_L/\phi(\Delta_L)}{\log{x}}. \]
\end{enumerate}
Although $x_0$ may depend on $\Ll$, all implied constants in this statement are absolute.
\end{prop}
\begin{proof}[Proof (sketch)] This follows from \cite[Proposition 6.1]{maynard14d}. In the setup of that proposition, $\A$ is the set of natural numbers, $\Ll$ is as above, $\curly{P}$ is the set of all primes, $B=1$, $\theta = 2/5$, and $\alpha=1$. The probability measure on $\A(x)$ assigns to each $n$ the probability mass $w(n)/\sum_{n \in \A(x)} w(n)$. Our (1) follows from Proposition 6.1(1) together with the immediately preceding estimate for $w_n$; we also use Maynard's lower bounds on $\mathfrak{S}(\Ll)$ and $I_k(F)$ given in (8.2) and Lemma 8.6, respectively. Our (2) is deduced from Proposition 6.1(1,2); here we use the estimate  $J_k/I_k \gg \log{k}/k$ and the observation that for each $L(n) = a_L n + b_L \in \Ll$, we have (in Maynard's notation)
\[ \#\curly{P}_{L,\A}(x) = \sum_{\substack{a_L x+b_L \leq p < 2a_L x+b_L \\ p \equiv b_L\pmod{a_L}}} 1 \sim \frac{1}{\phi(a_L)} \frac{a_L x}{\log{x}} \sim \frac{a_L}{\phi(a_L)}\ \frac{\#\A(x)}{\log{x}}. \]
Our (3) comes from Proposition 6.1(1,4), and (4) comes from Proposition 6.1(1,3).
\end{proof}

We now prove Theorem \ref{thm:main}.

\begin{proof} Assume that $r \ge \exp(C_2 m)$, and let $\kappa = \lceil \exp(C_3 m)\rceil$. Let $c$ be a small positive absolute constant. The necessary constraints on the constants $C_2$, $C_3$, and $c$ will emerge in the proof.

Let $q_1', \dots, q_r'$ be the integers obtained from $q_1, \dots, q_r$ by replacing each $q_i$ with its squarefree part. That is, $q_i'$ is the unique squarefree integer for which $q_i/q_i'$ is a square. Since $q_1, \dots, q_r$ satisfy (*), so do $q_1', \dots, q_r'$. Let $k=2\kappa$, and let $L_i$ and $\tilde{L}_i$, for $1 \leq i \leq \kappa$, be the linear functions produced by  Lemma \ref{lem:admissible} applied with $q_1', \dots, q_r'$. Every prime dividing $q_1' \cdots q_r'$ divides $f:=\mathfrak{f}(\Q(\sqrt{q_1}, \dots, \sqrt{q_r})/\Q)$, and thus $v = 16 \prod_{\ell \mid q_1'\dots q_r',~\ell >2}\ell$ divides $16f$.

We now invoke Proposition \ref{prop:maynard}. We will show that with positive probability, an $n \in \A(x)$ satisfies all of
\begin{enumerate}
\item[(i)] at least $m$ of $L_1(n)$, \dots, $L_\kappa(n)$ are prime,
\item[(ii)] $p^{-}(L_i(n)) \ge x^{\frac{c}{k^3\log{k}}}$ and $p^{-}(\tilde{L}_i(n)) \ge x^{\frac{c}{k^3\log{k}}}$ for all $i=1,\dots, \kappa$,
\item[(iii)] all integers in the interval $[L_1(n),L_\kappa(n)]$ that are not one of the $L_i(n)$ are composite,
    \item[(iv)] whenever $p=L(n)$ is prime with $L \in \{L_1, \dots, L_{\kappa}\}$, $p$ possesses some element of $\Qq$ as a primitive root.
\end{enumerate}
If (i)--(iv) hold for $n$, then the set of primes in $[L_1(n), L_{\kappa}(n)]$ has at least $m$ elements, each one of which possesses some element of $\Qq$ as a primitive root. Moreover, the difference between the largest and smallest such primes is at most
\[  L_{\kappa}(n) - L_1(n) = a_{\kappa}-a_1 \leq v (2\kappa)^{C_1} \leq f \exp(C_4 m), \]
provided that $C_4$ is large enough in terms of $C_1$ and $C_3$.  Thus, we obtain Theorem \ref{thm:main} with $C = C_2$ and $C' = C_4$.

To begin analyzing (i)--(iv), let $\Pp$ be the set of primes, and consider the random variable $X := \sum_{i=1}^{\kappa} \1_{\Pp}(L_i(n))$. Proposition \ref{prop:maynard}(2) and our choice of $\kappa$ yield $\E[X] \gg C_3 m$. We assume $C_3$ is large enough that $\E[X] > m$. Noting the inequality \[ \1_{X\geq m} \geq \kappa^{-1}(X-(m-1)), \] and taking expectations, we find that (i) holds with probability at least $\kappa^{-1}$.

Let $L$ be one of the linear functions in $\Ll$. Then \[\Prob(p^{-}(L(n)) < x^{c/(k^3\log{k})})\le \E\bigg[\sum_{p \mid L(n),~p < x^{c/(k^3\log{k})}}1\bigg].\] So from Proposition \ref{prop:maynard}(3), (ii) fails with probability
\[ \ll k \cdot \left(\frac{c}{k^3\log{k}}\right)^2 k^4 (\log{k})^2. \]
We may assume $c>0$ is small enough that the odds of failure are less than $\frac{1}{2}\kappa^{-1}$. Then (i) and (ii) hold simultaneously with probability at least $\frac{1}{2}\kappa^{-1}$.

We claim that (iii) fails with probability $o(1)$, as $x\to\infty$. It is enough to show that if $a$ is a fixed integer from $[a_1, a_\kappa]$, and $a \not\in\{a_1,\dots, a_\kappa\}$, then the probability that $L(n):=vn+a$ is prime is $o(1)$. This follows immediately from Proposition 6.1(4) if $L$ is not a rational multiple of any $L_i$ or $\tilde{L}_i$. Since $L$ has leading coefficient $v$ and $a\not\in \{a_1,\dots, a_\kappa\}$, $L$ is not a multiple of any $L_i$. Since each $\tilde{L}_i$ has leading coefficient $v/T$, if $L$ is a multiple of some $\tilde{L}_i$, then $L = T \tilde{L_i}$; but then $T \mid L(n)$ and so $L(n)$ is  composite for all $n\in \A(x)$.

Now assume that (ii) holds but that (iv) fails. We will show that this occurs with probability $o(1)$, as $x\to\infty$. In view of our previous estimates, this will complete the proof of Theorem \ref{thm:main}.

Assume that $p=L(n)$ is prime, with $L \in \{L_1, \dots, L_{\kappa}\}$, but that $p$ fails to have any $q\in \Qq$ as a primitive root. From Lemma \ref{lem:HB}(1) and our choice of $\Ll$, each $q \in \Qq$ is is a nonsquare modulo $p$. Thus, for each $q \in \Qq$, there is a prime $s=s_q$ dividing $(p-1)/T$ for which $q$ is an $s$th power modulo $p$. Put $t= \lceil 2c^{-1} k^3\log{k}\rceil$. Since (ii) holds, $\Omega((p-1)/T) \leq t$. (We assume here, as elsewhere in the proof, that $x$ is sufficiently large.) Recalling that $k=2\lceil \exp(C_3m)\rceil$, we  assume $C_2$ is large enough that
\[ \exp(C_2 m) > (t-1)t. \]
Since $\#\Qq = r \geq \exp(C_2 m)$, the pigeonhole principle guarantees that at least $t$ values of $q \in \Qq$ share the same value of $s_q$; call this common value $s$. Relabeling, we can assume these are $q_1,\dots, q_t$. Then
\[ \#\langle q_1, \dots, q_t \bmod{p}\rangle \leq \frac{p-1}{s}\le \frac{p-1}{p^{-}((p-1)/T)}\leq x^{1-\frac{c}{2 k^3\log{k}}} \le x^{1-\frac{1}{t}}. \]
By Lemma \ref{lem:MS}, $p$ is restricted to a set of size $\ll_{\Qq} (x^{1-1/t})^{1+1/t} = x^{1-\frac{1}{t^2}}$.
Given $L$, the prime $p=L(n)$ determines $n$, restricting $n$ also to a set of size $O_{\Qq}(x^{1-1/t^2})$. Since there are $O_m(1)$ possibilities for $L$, the number of $n\in \A(x)$ for which (ii) holds but (iv) fails is $O_{\Qq,m}(x^{1-1/t^2})$.

Since $n$ satisfies (ii), each of $L_1(n),\dots,L_{\kappa}(n)$ has at most $t$ prime factors. So from Proposition \ref{prop:maynard}(i), the probability mass at $n$ is at $O_m(x^{-1} (\log{x})^{k})$. Thus, the probability of selecting an $n$ detected in the previous paragraph is
$O_{m,\Qq}((\log{x})^k x^{-1/t^2})$, which is $o(1)$ as $x\to\infty$.
\end{proof}

\section{Preparation for the proof of Theorem \ref{thm:EC}}
We begin with some background on elliptic curves. For each prime $\ell$, let $K_\ell$ denote the $\ell$-torsion field $\Q(E[\ell])$. It is well-known and easy to check that $K_\ell$ is a Galois extension of $\Q$. Now let $p$ be a prime of good reduction for $E$. Clearly, $E(\F_p)$ is cyclic if and only if it does not contain $(\Z/\ell\Z)^2$ for any prime $\ell$.
The following lemma, due to Ram Murty \cite[p. 159]{murty83}, shows that whether or not $E(\F_p)$ is cyclic amounts to a series of conditions on the splitting of $p$ in the fields $K_\ell$.

\begin{lem}\label{lem:murty} Let $p$ be a prime of good reduction for $E$. If $\ell$ is a prime with $\ell \neq p$, then $E(\F_p)$ contains $(\Z/\ell\Z)^2$ if and only if $p$ splits completely in $K_{\ell}$. As a consequence, $\E(\F_p)$ is cyclic if and only if for all primes $\ell \neq p$,
\begin{equation}\label{eq:notsplit} \text{$p$ does not split completely in $K_{\ell}$}. \end{equation}
\end{lem}

\begin{remark} If $E(\F_p)$ contains $(\Z/\ell\Z)^2$, then $\ell^2 \mid \#E(\F_p) \leq (\sqrt{p}+1)^2$, and so it suffices to test \eqref{eq:notsplit} for
\begin{equation}\label{eq:sufficetotest} \ell \le \sqrt{p}+1. \end{equation}
\end{remark}

If we assume the GRH, then the following theorem of Lagarias and Odlyzko \cite{LO77} gives a satisfactory estimate for the frequency with which primes split completely in $K_\ell$. We state the result incorporating a small improvement by Serre \cite[\S2.4]{serre81}.

\begin{prop}[Effective Chebotarev theorem, on GRH]\label{thm:serre}  Let $K$ be a finite Galois extension of $\Q$, and let $\C$ be a conjugacy class of $\Gal(K/\Q)$. The number of unramified primes $p \leq x$ with $\legq{K/\Q}{p} = \C$ is given by
	\[ \frac{\#\C}{[K:\Q]}\Li(x)+ O\left(\#\C \cdot x^{1/2}\left(\frac{\log|\Delta_K|}{[K:\Q]} + \log{x}\right)\right), \]
for all $x\geq 2$. Here the $O$-constant is absolute.
\end{prop}

\begin{remark} To estimate the $O$-term, we will use the following estimate valid for any Galois extension $K/\Q$ (see \cite[Proposition 6]{serre81}):
\begin{equation}\label{eq:hensel} \frac{1}{[K:\Q]} \log|\Delta_K| \leq \log~[K:\Q]+ \sum_{p \mid \Delta_K} \log{p}. \end{equation}
\end{remark}

To apply \eqref{eq:hensel}, we need to understand which primes ramify in $K_\ell$. The following result can be derived from a criterion of N\'eron--Ogg--Shafarevich \cite[Theorem 7.1, p. 201]{silverman09}.

\begin{lem} Let $E/\Q$ be an elliptic curve. Every prime that ramifies in $K_{\ell}$ divides $\ell \cdot \Delta_{E}$.
\end{lem}

We will find bounded gaps among primes $p$ produced by certain linear functions, with coefficients chosen to give $p$ a ``leg up'' in terms of $E(\F_p)$ being cyclic. To build these functions, we need the following analogue of Lemma \ref{lem:HB}.

\begin{lem}\label{lem:HB2} Let $M$ be either a quadratic or abelian cubic extension of $\Q$. Let $f = \mathfrak{f}(M/\Q)$, and let $v= 2^4 3^3 \prod_{\ell \mid f,~\ell > 3} \ell$. One can select an integer $u$ coprime to $v$ so that both of the following hold:
\begin{enumerate}
\item[\normalfont\rmfamily{(1)}] For every prime $p\equiv u\pmod{v}$, $p$ is inert in $M$.
\item[\normalfont\rmfamily{(2)}] If $T$ is the largest power of $2$ dividing $u-1$, then $T \in \{2,4,8\}$, and $\gcd(\frac{u-1}{T},v)=1$.
\end{enumerate}
\end{lem}

\begin{proof} We make free use of the correspondence between abelian extensions of $\Q$ and groups of primitive Dirichlet characters, as reviewed in \cite[Chapter 3]{washington97}.

If $M/\Q$ is quadratic, then $f$ is the absolute value of a fundamental discriminant, whereas if $M/\Q$ is abelian cubic, then
\[ f = 9 q_1 \cdots q_k \quad\text{or}\quad f = q_1\cdots q_k,\quad\text{for distinct primes}\quad q_i \equiv 1\pmod{6}. \]
Thus, $2^4 \nmid f$, $3^3\nmid f$, and every prime $\ell > 3$ that divides $f$ appears to the first power only. Let $H$ be the subgroup of $\Gal(\Q(\zeta_f)/\Q)$ that fixes $M$. We identify $\Gal(\Q(\zeta_f)/\Q)$ with $(\Z/f\Z)^{\times}$. Note that $H$ has index $[M:\Q] > 1$. Since $M$ is cyclic of prime degree, an unramified prime $p$ either remains inert or splits completely, the latter holding exactly when $p\bmod{f} \in H$.

Choose an integer $u_0$ with
\begin{equation}\label{eq:u00} \gcd(u_0,f)=1, \quad  u_0\bmod{f} \not \in H, \quad\text{and}\quad u_0 \equiv 2\pmod{3}.\end{equation} This is clearly possible if $3\nmid f$. If $3\mid f$, we argue by contradiction: If there is no such $u_0$, then
$\#H > \#\{1 \leq h \leq f: \gcd(h,f)=1, h\equiv 2\pmod{3}\} = \frac{1}{2}\phi(f)$, where the inequality is strict since $1\bmod{f} \in H$. But then $H = \Gal(\Q(\zeta_f)/\Q)$, a contradiction.

We can also assume that
\begin{equation}\label{eq:u01} u_0 \equiv 1\pmod{2}. \end{equation}
Indeed, if $f$ is even, this condition is automatic, whereas if $f$ is odd but $u_0$ is even, we can replace $u_0$ by $u_0 + 3f$. Finally, we can assume that
\begin{equation}\label{eq:u02} 16 \nmid u_0-1, \end{equation}
by replacing $u_0$ with $u_0 + \lcm[24,f]$ if necessary.

If $M/\Q$ is quadratic, then for each prime $\ell > 3$ dividing $f$, put
\[ u_{\ell} = \begin{cases} u_0 &\text{if $\ell \nmid u_0-1$}, \\
4u_0&\text{otherwise}.
\end{cases} \]
Then $\ell \nmid u_{\ell}-1$.  If $M/\Q$ is abelian cubic, then for each prime $\ell > 3$ dividing $f$, put
\[ u_{\ell} = \begin{cases} u_0&\text{if $\ell \nmid u_0-1$}, \\
-8u_0&\text{otherwise}.
\end{cases} \]
In this case, we again have that $\ell \nmid u_{\ell}-1$. Finally, select $u$ so that
\[ u \equiv u_0 \pmod{2^4\cdot 3^3}\quad\text{and}\quad u\equiv u_{\ell}\pmod{\ell}~\text{$\forall \ell \mid f$ with $\ell > 3$}. \]
This puts $u$ in a well-defined coprime residue class modulo $v$.

We now check (1) and (2). In the case when $M/\Q$ is quadratic, $u\equiv u_0 g^2\pmod{f}$ for some integer $g$. Since $H$ has index $2$, $g^2\bmod{f} \in H$. Since $u_0 \bmod{f}\not\in H$, we find that $u\bmod{f}\not\in H$. So if $p\equiv u\pmod{v}$, then $p\bmod{f} \not\in H$ (notice $f \mid v$) and so $p$ is inert in $M$. An analogous argument works when $M/\Q$ is abelian cubic; in that case, $H$ has index $3$ and $u\equiv u_0 g^3\pmod{f}$ for some $g$. This completes the verification of (1). Since $u\equiv u_0\pmod{16}$, \eqref{eq:u01} and \eqref{eq:u02} yield $T \in \{2,4,8\}$. Since $u\equiv u_0\pmod{3}$, \eqref{eq:u00} shows that $3\nmid u-1$. For each prime $\ell > 3$ dividing $v$, our choices of $u_{\ell}$ ensure that $\ell \nmid u-1$.  Hence, $\gcd(\frac{u-1}{T},v)=1$, which completes the proof of (2).
\end{proof}

By imitating the deduction of Lemma \ref{lem:admissible} from Lemma \ref{lem:HB}, we obtain the following consequence of Lemma \ref{lem:HB2}.

\begin{lem}\label{lem:admissible2}  Let $M$ be either a quadratic or abelian cubic extension of $\Q$. Let $u$ and $v$ be chosen as in Lemma \ref{lem:HB2}. Let $\kappa$ be a natural number. There are integers $a_1 < \dots <  a_{\kappa}$, each congruent to $u\bmod{v}$, for which the $2\kappa$ linear functions
\begin{alignat*}{3} L_1(n) &= vn + a_1, &\quad \dots, \quad &L_{\kappa}(n) = vn + a_{\kappa}, \\
\tilde{L}_1(n) &= \frac{v}{T}n + \frac{a_1-1}{T}, &\quad \dots, \quad &\tilde{L}_{\kappa}(n) = \frac{v}{T} n + \frac{a_{\kappa}-1}{T} \end{alignat*}
make up an admissible family. Moreover, we can select the $a_{i}$ in such a way that
\[ a_{\kappa} - a_1 \leq v\cdot (2\kappa)^{C_5}. \]
\end{lem}

\section{Proof of Theorem \ref{thm:EC}}
\subsection{The GRH case}
By assumption, $K_2 \neq \Q$. Since $K_2$ is the splitting field of a cubic polynomial, it has a subfield $M$ that is either quadratic or abelian cubic over $\Q$. Let $\kappa = \lceil\exp(C_6 m) \rceil$, where $C_6$ is a large absolute constant. Let $k=2\kappa$, and let $\Ll$ consist of the linear functions $L_1, \dots, L_{\kappa}$, $\tilde{L}_1, \dots, \tilde{L}_{\kappa}$ constructed in Lemma \ref{lem:admissible2}. Recall that each $L_i$ has leading coefficient $v=2^4 3^3 \prod_{\ell \mid f,~\ell > 3}\ell$, where $f$ is the conductor of $M$. If $\ell \mid f$, then $\ell \mid \Delta_{K_2}$, and so $\ell =2$ or $\ell$ is a prime of bad reduction. Consequently,
\[ v \mid 2^4 3^3 \cdot \rad(\Delta_{E}). \]
We warn the reader of the following innocuous abuse of notation: If $L=L_i$, we will write $\tilde{L}$ for $\tilde{L}_i$.

Assume $x$ is large. We will show that if $c>0$ is a sufficiently small absolute constant, then  with positive probability, an $n \in \A(x)$ satisfies all of
\begin{enumerate}
\item[(i)] at least $m$ of $L_1(n)$, \dots, $L_\kappa(n)$ are prime,
\item[(ii)] $p^{-}(L_i(n)) \ge x^{\frac{c}{k^3\log{k}}}$ and $p^{-}(\tilde{L}_i(n)) \ge x^{\frac{c}{k^3\log{k}}}$ for all $i=1,\dots, \kappa$,
\item[(iii)] all integers in the interval $[L_1(n),L_\kappa(n)]$ that are not one of the $L_i(n)$ are composite,
\item[(iv)] if $p=L(n)$ is prime with $L \in \{L_1, \dots, L_{\kappa}\}$, then $p$ is inert in every $K_\ell$ with $\ell> x^{1/3}$ and $\ell \neq p$,
\item[(v)] if $p=L(n)$ is prime with $L \in \{L_1, \dots, L_{\kappa}\}$, then $E(\F_p)$ is cyclic.
\end{enumerate}
If all of (i)--(v) hold for $n$, then the set of primes $p \in [L_1(n), L_\kappa(n)]$ has at least $m$ elements, all of these have $E(\F_p)$ cyclic, and the gap between the largest and smallest is at most
\[ L_{\kappa}(n)-L_1(n) \leq v \cdot (2\kappa)^{C_5} \leq \rad(\Delta_E) \cdot \exp(C_7 m); \]
the GRH half of Theorem \ref{thm:EC} follows.

For the sake of readability, in the remainder of the proof we suppress the dependence of implied constants on $E$.

To handle (i)--(iii), we proceed as in the proof of Theorem \ref{thm:main}. Arguments given there show that if we fix $C_6$ sufficiently large and $c$ sufficiently small, then (i) and (ii) hold simultaneously with probability at least $\frac{1}{2}\kappa^{-1}$, while (iii) fails with probability $o(1)$, as $x\to\infty$.

Now suppose that (i)--(iii) hold for $n$ but that (iv) fails. We will show that this occurs with probability $o(1)$. Observe that for each $n \in [x,2x)$ and each $L \in \{L_1, \dots, L_\kappa\}$, the integer $L(n)$ is smaller than $3vx$.

We start by bounding the number of  $p \leq 3vx$ which split completely in $K_\ell$ for some $\ell > x^{1/3}$ with $\ell \neq p$.  In that case, $(\Z/\ell\Z)^2$ sits inside $E(\F_p)$, and so
\begin{equation}\label{eq:firstdiv} \ell^2 \mid p+1-a_p. \end{equation}
Since $\Q(\zeta_l)\subset K_\ell$ (by properties of the Weil pairing \cite[Corollary 8.1.1]{silverman09}),
\begin{equation}\label{eq:2nddiv} \ell \mid p-1. \end{equation}
Comparing \eqref{eq:firstdiv} and \eqref{eq:2nddiv} shows that $\ell \mid 2-a_p$. If $a_p \neq 2$, then
\[ 0 < |2-a_p| < 2+2\sqrt{3vx} < x^{2/3} < \ell^2; \]
hence $\ell$ is uniquely determined by $a=a_p$, as the largest prime dividing $|2-a|$. Fixing $a\ne 2$, \eqref{eq:firstdiv} shows that the number of corresponding $p\le 3vx$ is $\ll \frac{x}{\ell^2} + 1 \ll x^{1/3}$. By the Hasse bound, $|a| \ll \sqrt{x}$, and so summing on the possible values of $a$ shows that $O(x^{5/6})$ values of $p$ arise in this way. On the other hand, when $a=2$,  \eqref{eq:sufficetotest} and \eqref{eq:firstdiv} imply that the number of corresponding $p$ is
\[ \ll  \sum_{\ell \in (x^{1/3}, \sqrt{3vx}+1]} \left(\frac{x}{\ell^2}+1\right) \ll x^{2/3}. \]
So there are a total of $O(x^{5/6})$ of these primes $p$.

Since (i)--(iii) hold while (iv) fails, there is an $L \in \{L_1, \dots, L_\kappa\}$ such that $p=L(n)$ is among the primes counted in the previous paragraph. There are $O_m(1)$ possibilities for $L$, and so $O_m(x^{5/6})$ possibilities for $n\in\A(x)$. From (ii) and Proposition \ref{prop:maynard}(1), the probability mass at each such $n$ is $O_m(x^{-1} (\log{x})^k)$. So the probability that (i)--(iii) hold but (iv) fails is $O_m(x^{-1/6} (\log{x})^k)$, which is $o(1)$ as $x\to\infty$.

To complete the proof, we show the probability (i)--(iv) hold but (v) fails is also $o(1)$.

Suppose $p=L(n)$ is prime, with $L \in \{L_1, \dots, L_\kappa\}$, but that $E(\F_p)$ is not cyclic. From Lemma \ref{lem:HB2}(1) and our choice of $\Ll$, $p$ is inert in $M$, and a fortiori does not split completely in $K_2$. So $\E(\F_p)$ must split completely in $K_\ell$ for some $\ell > 2$. Since $\ell \mid p-1 = T \cdot \tilde{L}(n)$, (ii) and (iv) imply that
\begin{equation}\label{eq:lrange} x^{c/(k^3\log{k})} \le \ell \leq x^{1/3}. \end{equation}

We now count how many $p \leq 3vx$ split completely in $K_\ell$ for some $\ell$ in the range \eqref{eq:lrange} with $\ell\neq p$. Making the same appeal to Proposition \ref{prop:maynard}(1) we saw earlier in the proof, it is enough to prove that the number of these $p$ is
\begin{equation}\label{eq:suffices} \ll x^{1-\frac{c}{k^3\log{k}}} + x^{5/6} \log{x}. \end{equation}

 We invoke GRH. By effective Chebotarev, the number of $p\leq 3vx$ splitting completely in $K_\ell$ is
\[ \ll \frac{x}{[K_\ell:\Q] \log{x}} + x^{1/2}\left(\log{x}+ \frac{1}{[K_\ell:\Q]} \log|\Delta_{K_\ell}|\right). \]
Since every prime dividing $\Delta_{K_\ell}$ divides $\ell \cdot \Delta_E$, \eqref{eq:hensel} shows that this upper bound is
\begin{equation}\label{eq:upper2} \ll \frac{x}{[K_\ell:\Q] \log{x}} + x^{1/2} \log([K_\ell:\Q]\cdot  \ell x ). \end{equation}
If $E$ has CM, then for all large primes $\ell$, the degree of $K_\ell/\Q$ is either $2(\ell-1)^2$ or $2(\ell^2-1)$, according to whether or not $\ell$ splits in the CM field. In particular, $[K_\ell:\Q] \asymp \ell^2$. If the non-CM case, we have $[K_\ell:\Q]= \#\mathrm{GL}_2(\Z/\ell\Z) \asymp \ell^4$ for all large primes $\ell$. (These results are due to Serre \cite{serre72}; see \cite[Theorem 18]{CCS13} for a detailed discussion of the CM case.) Thus, the sum of \eqref{eq:upper2} over the range \eqref{eq:lrange} is
\[ \ll \frac{x}{\log{x}} \sum\frac{1}{\ell^2} + x^{1/2}\log{x}\sum_{\ell} 1 \ll x^{1-\frac{c}{k^3\log{k}}} + x^{5/6}\log{x}, \]
which agrees with \eqref{eq:suffices}. This completes the proof in the GRH case.

\subsection{Unconditional proof in the CM case} As already mentioned in the introduction, we will deal entirely with supersingular primes in this part of the proof.

Suppose that $p\ge 5$ is supersingular but that $E(\F_p)$ is not cyclic. Choose an $\ell \neq p$ for which $p$ splits completely in $K_\ell$. Then $\ell^2 \mid \#E(\F_p) = p+1$ and $\ell \mid p-1$, forcing $\ell=2$. Consequently, $p$ splits in the quadratic or abelian cubic subfield $M$ of $K_2$.

Let $F$ be the CM field. We look for primes $p$ of good reduction that are inert in $F$ --- guaranteeing that $p$ is supersingular --- and inert in $M$. If $F=M$, it is clear that there are infinitely many such primes; otherwise, this follows from the linear disjointness of $F$ and $M$ over $\Q$. Since $F/\Q$ and $M/\Q$ are abelian, the set of such $p$ contains all primes in a certain arithmetic progression modulo $q:=f_1 f_2$, where $f_1=\mathfrak{f}(F/\Q)$ and $f_2=\mathfrak{f}(M/\Q)$. Since $E$ is defined over $\Q$, its CM field $F$ must be one of the nine imaginary quadratic fields of class number $1$ (see, e.g., Serre's chapter in \cite{CF86}), and so $f_1 \ll 1$. On the other hand, since every odd prime dividing $f_2$ divides $\Delta_E$, and since $f_2$ is squarefree apart from bounded powers of $2$ and $3$, the modulus $q = f_1 f_2 \ll f_2 \ll \rad(\Delta_E)$.

Corollary 3 of \cite{BFTB14} asserts that for any fixed coprime progression mod $q$, there are infinitely many tuples of $m$ consecutive primes $p_1 < p_2 < \dots < p_m$ with $p_m - p_1 \ll_{q,m} 1$. In fact, it is straightforward to modify their argument to get an upper bound of $q \exp(O(m))$ (cf.  \cite[Theorem 2(2)]{thorner14} when $m=2$). The theorem follows.

\begin{remark} In the non-CM case, we do not have an unconditional bounded gaps result for primes $p$ with $E(\F_p)$ cyclic. But if `cyclic' is replaced by `has an element of order $> p^{3/4-\epsilon}$', then such a result  follows quickly from work of Duke \cite{duke03}.

Let $E/\Q$ be any elliptic curve. (No assumption on the rational torsion is needed here.) For each prime $p$ of good reduction, write $E(\F_p) \cong \Z/d_p\Z \oplus \Z/e_p\Z$ for natural numbers $d_p$ and $e_p$ where $d_p \mid e_p$. Clearly, $d_p^2 \le \#E(\F_p) \leq (\sqrt{p}+1)^2$, so that $d_p \leq 2\sqrt{p}$.

Duke shows (see \cite[eq. (8)]{duke03}) that for each $n \le 2\sqrt{x}$, the number of $p \leq x$ for which $n \mid d_p$ is $O(x^{3/2} n^{-3})$. A fortiori, the same bound holds for how often $d_p=n$. Consequently, the number of $p\le x$ with $d_p > x^{1/4+\epsilon/2}$ is $O(x^{1-\epsilon})$. Whenever $d_p \leq x^{1/4+\epsilon/2}$, the group $E(\F_p)$ has an element of order
\[ e_p \ge \frac{\#E(\F_p)}{x^{1/4+\epsilon/2}} \gg p x^{-\frac14-\frac{\epsilon}{2}}. \]
Summing dyadically, we conclude that $E(\F_p)$ has an element of order $> p^{\frac{3}{4}-\epsilon}$ for all but $O_\epsilon(x^{1-\epsilon})$ primes $p \le x$. This exceptional set is so sparse that it follows immediately from Maynard's lower bound results (see \cite[Theorem 3.1]{maynard14d}) that the set of nonexceptional $p$ has bounded gaps. More precisely, this set contains arbitrarily long runs of primes contained in bounded length intervals.\end{remark}

\section{Proof of Theorem \ref{thm:CM}}
We begin by stating a variant of Proposition \ref{prop:maynard} for sets of primes described by Chebotarev conditions.

\begin{prop}\label{prop:maynard2} Let $K/\Q$ be a Galois extension, and let $\C$ be a fixed conjugacy class of $\Gal(K/\Q)$. Let \[ \Pp(\C)= \{p: p \nmid \Delta_K, \legq{K/\Q}{p} = \C\}.  \] Suppose $a_1 < a_2 < \dots < a_\kappa$ are odd integers for which the $k=2\kappa$ linear functions
\begin{multline}\label{eq:simpleadmissible} L_1(n) = 2n+a_1, \quad L_2(n) = 2n+a_2, \quad\dots,\quad L_\kappa(n) = 2n+a_\kappa, \\
\tilde{L}_1(n) = n+\frac{a_1-1}{2}, \quad \tilde{L}_2(n) = n+\frac{a_2-1}{2}, \quad\dots,\quad \tilde{L}_\kappa(n) = n+\frac{a_\kappa-1}{2}
\end{multline}
form an admissible collection; call this collection $\Ll$.
Suppose that $x$ is sufficiently large, $x > x_0(K,\Ll)$. There is a probability measure on $\A(x)= \{n \in \Z: x \leq n < 2x\}$
with all of the following properties:
\begin{enumerate}
\item[\normalfont\rmfamily{(1)}] The probability mass at any single $n \in \A(x)$ is \[ \ll_{K} x^{-1} (\log{x})^k \left(\prod_{i=1}^{k} \prod_{\substack{p \mid L_i(n) \\ p \nmid 2 \Delta_K}} 4\right) \exp(O(k\log{k})). \]
\item[\normalfont\rmfamily{(2)}] For each $L \in \Ll$,
\[ \Prob(L(n)\text{ belongs to $\Pp(\C)$})\gg_K \frac{\log{k}}{k}. \]
\item[\normalfont\rmfamily{(3)}] Let $\rho \in [k\frac{(\log\log{x})^2}{\log{x}}, \frac{1}{30 [K:\Q]}]$. For each $L \in \Ll$,
    \[ \E\bigg[\sum_{\substack{p \mid L(n) \\ p \leq x^{\rho}, ~p\nmid 2\Delta_K}} 1\bigg] \ll \rho^2 k^4 (\log{k})^2. \]
\end{enumerate}
The implied constant in (3) is absolute.
\end{prop}
\begin{proof}[Proof (sketch)] The main technical input is supplied by a variant of the Bombieri--Vinogradov theorem due to Murty and Murty \cite{MM87}, which asserts that $\Pp(\C)$ has level of distribution $\theta$ for any fixed \[ \theta < \min\{\frac{1}{2}, \frac{2}{[K:\Q]}\};\] here the moduli of the arithmetic progressions are assumed coprime to $\Delta_K$. We now argue as in the proof of Proposition \ref{prop:maynard}. Specifically, the Murty--Murty theorem allows us to apply \cite[Proposition 6.1]{maynard14d} with $\A=\N$, $\Ll$ as given, $\curly{P}= \Pp(\C,K)$, $B=2 \Delta_K$, $\theta = \min\{\frac{1}{3}, \frac{1}{[K:\Q]}\}$, and $\alpha=1$. Defining the probability mass at $n$ as $w(n)/\sum_{n \in \A(x)} w(n)$, the result follows. (For similar applications of the Murty--Murty theorem, see \cite{thorner14} and \cite[Theorem 3.5]{maynard14d}.)
\end{proof}

The proof of Theorem \ref{thm:CM} also uses the following criterion, which is contained in work of Cojocaru \cite[Lemmas 2.2 and 2.3]{cojocaru03}.

\begin{lem}\label{lem:cojocaru} Suppose that $E/\Q$ has CM by an order in the imaginary quadratic field $F$. Let $p$ be a prime of good ordinary reduction, and let $\ell$ be a prime with $\ell \neq p$. If $p$ splits completely in $K_\ell$, then there is a $\pi \in \Z_F$ with $\pi \equiv 1\pmod{\ell}$ and $N(\pi)=p$.\end{lem}

\begin{proof}[Proof of Theorem \ref{thm:CM}] We begin by specifying the parameters needed for our application of Proposition \ref{prop:maynard2}.

Let $F$ be the CM field of $E$. Let $\Qq$ be the set of primes dividing $2\Delta_F \Delta_E$, and let $K$ be the compositum of $F$ and all of the fields $K_\ell:=\Q(E[\ell])$ for $\ell \in \Qq$. Then $K/\Q$ is Galois and every prime dividing $\Delta_K$ belongs to $\Qq$. Choose a conjugacy class $\C$ of $\Gal(K/\Q)$ where every prime $p \in \Pp(\C)$ is such that
\begin{itemize}
\item $p$ splits in $F$,
\item $p$ does not split completely in any of the fields $K_\ell$ with $\ell \in \Qq$.
\end{itemize}
Any large prime $p$ of ordinary reduction for which $E(\F_p)$ is cyclic satisfies both of these conditions, and so such a $\C$ must exist.

Let $\kappa = \lceil\exp(C_K m)\rceil$, where $C_K$ is a sufficiently large constant depending on $K$. Mimicking  the proof of Lemma \ref{lem:admissible}, we can choose odd integers $a_1 < \dots < a_{\kappa}$ for which \eqref{eq:simpleadmissible} is admissible, with $a_{\kappa} - a_1 \leq (2\kappa)^{C_8}$. Then \begin{equation}\label{eq:finalgapbound} a_{\kappa}-a_1 \ll \exp(O_E(m)).\end{equation}

We are now in a position to apply Proposition \ref{prop:maynard2}. If $C_K$ is sufficiently large and $c$ is sufficiently small (both allowed to depend on $K$), then an $n\in \A(x)$ satisfies both of the following conditions with probability $\gg_m 1$:
\begin{enumerate}
\item[(i)] at least $m$ of $L_1(n), \dots, L_\kappa(n)$ belong to $\Pp(\C)$,
\item[(ii)] whenever a prime $\ell \leq x^{c/(k^3\log{k})}$ divides $\prod_{i=1}^{\kappa}L_i(n) \tilde{L}_i(n)$,  $\ell$ also divides $2\Delta_K$.
\end{enumerate}
Indeed, this follows from arguments seen already in the proofs of Theorems  \ref{thm:main} and \ref{thm:EC}, the only difference being that we appeal to Proposition \ref{prop:maynard2} instead of Proposition \ref{prop:maynard}.  We now introduce the statement
\begin{enumerate}
\item[(iii)] Whenever $p = L(n) \in \Pp(\C)$, with $L \in \{L_1, \dots, L_\kappa\}$, the group $E(\F_p)$ is cyclic.
\end{enumerate}
We will show that the probability (i) and (ii) hold but (iii) fails is $o(1)$, as $x\to\infty$, so that (i)--(iii) hold with positive probability for all large $x$. This will complete the proof; indeed, if $n\in \A(x)$ satisfies (i)--(iii), and $p_1 < p_2 < \dots < p_m$ are primes from $\Pp(\C)$ drawn from $\{L_1(n), \dots, L_\kappa(n)\}$, then the claimed bound on $p_m-p_1$ follows from \eqref{eq:finalgapbound}, while the fact that each of the primes is of good ordinary reduction follows from the choice of $\C$.

Suppose (i) and (ii) hold and that $p = L_i(n) \in \Pp(\C)$, where $i\in \{1,2,\dots, \kappa\}$. As we have just remarked, $p$ is a prime of good ordinary reduction. If $E(\F_p)$ is not cyclic, then $p$ spits completely in $K_\ell$ for some $\ell \neq p$. Then $\ell \mid p-1 = 2\tilde{L}_i(n)$, so that either $\ell \mid 2\Delta_K$ or $\ell \ge x^{c/(k^3\log{k})}$. But if $\ell \mid 2\Delta_K$, then $\ell \in \Qq$, and so the choice of $\C$ guarantees that $p$ does not split completely in $K_\ell$. So it must be that $\ell \ge x^{c/(k^3\log{k})}$. Since $p \leq 5x$ for large $x$, we also have that $\ell \leq \sqrt{5x}+1$, after recalling \eqref{eq:sufficetotest}.

Let us count primes $p\le 5x$ of good ordinary reduction that split completely in $K_\ell$ for some $\ell\ne p$ with
\begin{equation}\label{eq:lrange2}
x^{c/(k^3\log{k})}\le \ell \le \sqrt{5x}+1.
\end{equation}
From Lemma \ref{lem:cojocaru}, there is a  $\pi_p \in \Z_{F}$ with $\pi_p\equiv 1\pmod{\ell}$ and $N(\pi_p)=p$. The number of $\pi \in \Z_F$ with $\pi \equiv 1\pmod{\ell}$ and $N(\pi) \le 5x$ is $O(\frac{x}{\ell^2} + 1)$, by an elementary lattice point counting argument (e.g., see \cite[Lemma 5]{murty83} or \cite[Lemma 2.6]{cojocaru03}). Summing on $\ell$ in the range \eqref{eq:lrange2} shows that the number of $p$ in question is \[ \ll x^{1-\frac{c}{k^3\log{k}}} + x^{1/2}.\]

If $n$ satisfies (i) and (ii), Proposition
\ref{prop:maynard2}(1) shows that the probability mass at $n$ is $O_{K,m}(x^{-1} (\log{x})^k)$. Consequently, the probability that $L(n)$ is one of the primes counted in the preceding paragraph is $o(1)$, as $x\to\infty$.
\end{proof}

\subsection*{Acknowledgments} We are indebted to Pete L. Clark for helpful conversations on the theory of elliptic curves. This research began while the second author enjoyed a very pleasant visit to Brigham Young University. He thanks the BYU mathematics department for their hospitality and the NSF for their support under award DMS-1402268.

\providecommand{\bysame}{\leavevmode\hbox to3em{\hrulefill}\thinspace}
\providecommand{\MR}{\relax\ifhmode\unskip\space\fi MR }
\providecommand{\MRhref}[2]{%
  \href{http://www.ams.org/mathscinet-getitem?mr=#1}{#2}
}
\providecommand{\href}[2]{#2}

\end{document}